\documentclass[10pt,a4paper]{amsart}
\usepackage[utf8]{inputenc}
\usepackage[normalem]{ulem}
\usepackage[dvipsnames]{xcolor}
\usepackage[margin=2.5cm]{geometry}
\usepackage[pdftex,hyperfootnotes=false,colorlinks=true,linkcolor=blue,
citecolor=purple,filecolor=magenta,urlcolor=cyan,hypertexnames=false]{hyperref}
\usepackage{mathabx}
\usepackage{stmaryrd}
\usepackage{mathtools}

\def\cS{\mathsf{S}}
\def\cW{\mathcal{W}}
\def\cQ{\mathcal{Q}}

\title[Topological recursion of monotone orbifold Hurwitz numbers]{Topological recursion for monotone orbifold Hurwitz numbers: a proof of the Do-Karev conjecture}

\author[R.~Kramer]{R.~Kramer}

\address[R.~Kramer]{Max-Planck-Institut f\"ur Mathematik, Vivatsgasse 7, 53111 Bonn, Germany}
\email{rkramer@mpim-bonn.mpg.de}

\author[A.~Popolitov]{A.~Popolitov}

\address[A.~Popolitov]{Moscow Institute for Physics and Technology, Dolgoprudny, Russia; Institute for Information Transmission Problems, Moscow 127994, Russia; and ITEP, Moscow 117218, Russia}
\email{popolit@gmail.com}

\author[S.~Shadrin]{S.~Shadrin}

\address[S. Shadrin]{Korteweg-de Vriesinstituut voor Wiskunde, 
	Universiteit van Amsterdam, Postbus 94248,
	1090GE Amsterdam, The Netherlands}
\email{s.shadrin@uva.nl}


\makeatletter
\let\runauthor\@author
\let\runtitle\@title
\makeatother

\newcommand{\delete}{\bgroup\markoverwith{\textcolor{ForestGreen}{\rule[0.5ex]{2pt}{0.4pt}}}\ULon}
\newcommand{\deleteS}{\bgroup\markoverwith{\textcolor{green}{\rule[0.5ex]{2pt}{0.4pt}}}\ULon}
\newcommand{\mc}[1]{\mathcal{#1}}

\DeclareMathOperator{\Res}{Res}

\newtheorem{theorem}{Theorem}[section]
\newtheorem{proposition}[theorem]{Proposition}
\newtheorem{corollary}[theorem]{Corollary}
\newtheorem{lemma}[theorem]{Lemma}
\theoremstyle{definition}

\newtheorem{remark}[theorem]{Remark}

\begin{document}

\begin{abstract}
We prove the conjecture of Do and Karev that the monotone orbifold Hurwitz numbers satisfy the Chekhov-Eynard-Orantin topological recursion.
\end{abstract}

\maketitle

\tableofcontents

\section{Introduction} 

\subsection{Monotone orbifold Hurwitz numbers} A sequence of transpositions $\tau_1,\dots,\tau_m\in S_d$, $\tau_i=(a_i,b_i)$, $a_i<b_i$, $i=1,\dots,m$, is called monotone if $b_1\leq b_2\leq \cdots\leq b_m$.
For the entire paper, fix a positive integer $q$. The disconnected monotone $q$-orbifold Hurwitz numbers $h_{g,\mu}^\bullet$, $\mu=(\mu_1,\dots,\mu_\ell)$ are defined as 
\begin{equation}
h_{g,\mu}^\bullet \coloneqq \frac{|\mathrm{Aut}(\mu)|}{|\mu|!} \left| 
\left\{
(\tau_0,\tau_1,\dots,\tau_m) \, \Bigg|\,
\begin{array}{c}
\tau_i\in S_{|\mu|}, \tau_0\tau_1\cdots\tau_m \in C_\mu, \tau_0\in C_{(q,\dots,q)}, \\
m=2g-2+\ell+\frac{|\mu|}{q}, \text{ and} \\
\tau_1,\dots,\tau_m \text{ is a monotone sequence of transpositions}
\end{array}
\right\}
\right| 
\,.
\end{equation}
Here, $ |\mu | = \sum_{i=1}^\ell \mu_i$ and $ \mathrm{Aut} (\mu) = \{ \sigma \in S_\ell \mid \mu_j = \mu_{\sigma (j)} \forall j\} $.\par
The connected monotone Hurwitz numbers $h_{g,\mu}^\circ$ are defined by the same formula, but with an extra addition that $\tau_0,\tau_1,\dots,\tau_m$ generate a transitive subgroup of $S_{|\mu|}$.

The double monotone Hurwitz numbers were first introduced by Goulden, Guay-Paquet, and Novak, `double' allowing for $\tau_0$ to be any permutation, in~\cite{GouldenHCIZ} in their study of the HCIZ integral, and their orbifold, i.e. $ \tau_0 \in C_{(q,\dotsc,q)}$, version that we study in this paper was first considered explicitly as an object of research by Do and Karev in~\cite{DoKarev}. These numbers were very intensively studied in the recent years due to their rich system of connections to integrability, combinatorics, representation theory, and geometry, see e.g.~\cite{GouldenPolynomiality,GouldenetalGenus0,GuayHarnad,HarnadOrlov,ALS,HahnKramerLewanski,Hahn,ACEH}.

\subsection{Topological recursion} The topological recursion of Chekhov, Eynard, and Orantin~\cite{EynardOrantin} is a recursive procedure that associates to some initial data on a Riemann surface $\Sigma$ a sequence of meromorphic differentials $\omega_{g,n}$ on $\Sigma^{\times n}$. The initial data consist of $\Sigma$ itself, two non-constant meromorphic functions $x$ and $y$ on $\Sigma$, and a choice of a symmetric bi-differential $B$ on $\Sigma^{\times 2}$ with a double pole with bi-residue 1 on the diagonal.

We assume that $x$ has simple critical points $p_1,\dots,p_s\in\Sigma$, and by $\sigma_i$ we denote the local deck transformation for $x$ near the point $p_i$.  We also assume that the $ p_i$ are not critical points of $y$. We use the variables $z_i$ as the placeholders for the arguments of the differential forms to stress dependence on the point of the curve, and we denote by $z_I$ the set of variables with indices in the set $I$. Finally, $\llbracket n \rrbracket$ denotes the set $\{1,\dots, n\}$.

The topological recursion works as follows: first define $\omega_{0,1}\coloneqq ydx$, $\omega_{0,2}\coloneqq B$, and for $2g-2+n+1>0$ 
\begin{align}
\omega_{g,n+1}(z_0,z_{\llbracket n\rrbracket}) \coloneqq \frac 12\sum_{i=1}^s \mathop{\Res}\limits_{z\to p_i} \frac{\int_{z}^{\sigma_iz} B(\cdot, z_0)}{ydx(\sigma_iz)-ydx(z)}\Bigg[ \omega_{g-1,n+2}(z,\sigma_iz,z_{\llbracket n\rrbracket}) \\ \notag
\sum_{\substack{g_1+g_2=g,\ I_1\sqcup I_2 = \llbracket n \rrbracket \\ (g_1,|I_1|)\neq (0,0)\neq (g_2,|I_2|)}} \omega_{g_1,1+|I_1|}(z,z_I)\omega_{g_2,1+|I_2|}(\sigma_iz, z_{I_2}) \Bigg]\,.
\end{align}

Originally, this procedure was designed to compute the cumulants of some class of matrix models~\cite{ChekhovEynard}, but since then it has evolved a lot and nowadays it is intensively studied on the crossroads of enumerative geometry, integrable systems, and mirror symmetry, see e.g.~\cite{EynardBook,LiuMulase} for a survey of applications. In particular, it is the key ingredient of the so-called remodeling of the B-model conjecture proposed in~\cite{BKMP}, which suggests that topological recursion is the right version of the B-model for a class of enumerative problems, in the context of mirror symmetry theory.

\subsection{The Do-Karev conjecture} 
Denote by $H_{g,n}$ the $n$-point generating function for the connected $q$-orbifold monotone Hurwitz numbers:
\begin{equation}
H_{g,n}(x_1,\dots,x_n)\coloneqq \sum_{\mu_1,\dots,\mu_n=1}^\infty h^\circ_{g,\mu_1,\dots,\mu_n} \prod_{i=1}^n x_i^{\mu_i}\,.
\end{equation}

Consider the spectral curve data given by $\Sigma=\mathbb{C}$, $x(z)=z(1-z^{q})$ and $y(z)= z^{q-1}/(1-z^q)$, $B(z_1,z_2) = dz_1dz_2/(z_1-z_2)^2$ (our definition of $y$ differs by a sign from the one in~\cite{DoKarev} since we use a different sign in the definition of the recursion kernel than \emph{op.~cit.}). The critical points of $x(z)$ are $p_j = (q+1)^{-1/q} \exp(2\pi\sqrt{-1} j/q)$, $j=1,\dots,q$.

Consider the symmetric multi-differentials $\omega_{g,n}(z_1,\dots,z_n)$, $g\geq  0$, $n\geq 1$, defined on $\mathbb{C}^n$ by the Chekhov-Eynard-Orantin topological recursion. The conjecture of Do-Karev claims that
\begin{equation}
\omega_{g,n}(z_1,\dots,z_n) = d_1\otimes\cdots \otimes d_n H_{g,n}(x_1,\dots,x_n)\,,
\end{equation}
where we consider the Taylor series expansion near $x_1=\cdots=x_n=0$ and substitute $ x_i \to x(z_i)$. This conjecture is proved for $(g,n)=(0,1)$ in~\cite{DoKarev} and for $(g,n)=(0,2)$ in~\cite{KLS} and in an unpublished work of Karev. It is also proved in~\cite{DoDyerMathews,DBKraPS} for all $(g,n)$ in the case $q=1$. In this paper we prove it in the general case:
\begin{theorem} The conjecture of Do-Karev holds.
\end{theorem}

In addition to settling an explicitly posed open conjecture, this theorem is interesting in several different contexts. Firstly, it can be considered as a mirror symmetry statement in the context of the remodeling of the B-model principle of~\cite{BKMP}. Secondly, it is a part of a more general conjecture for weighted double Hurwitz numbers proposed in~\cite{ACEH} and its proof might be useful for the analysis of this more general conjecture. Thirdly, once the Do-Karev conjecture is proved, one can use the results of~\cite{Eynard,DOSS} to express the monotone orbifold Hurwitz numbers as the intersection numbers of the tautological classes on the moduli spaces of curves (for $q=1$, this is done in~\cite{ALS,DoKarev}).

\subsection{Proof} For the proof we use a corollary of~\cite[Theorem 2.2]{BorotShadrin} (see also~\cite{BorotEynardOrantin}). Namely, in order to prove that the differentials $d_1\otimes\cdots \otimes d_n H_{g,n}(x_1,\dots,x_n)$ satisfy the topological recursion on a given \emph{rational} spectral curve, it is sufficient to show that 
\begin{enumerate}
	\item The conjecture holds for $(g,n)=(0,1)$ and $(0,2)$.
	\item $H_{g,n}(x_1,\dots,x_n)$, $2g-2+n>0$, are the expansion at the point $x_1=\cdots=x_n=0$ of a finite linear combination of the products of finite order $d/dx_i$-derivatives of the functions $\xi_j(z_i)\coloneqq 1/(z_i-p_j)$, $x_i=x(z_i)$, $i=1,\dots,n$, $j=1,\dots,q$. 
	\item The differential forms $d_1\otimes\cdots \otimes d_n H_{g,n}(x_1,\dots,x_n)$, considered as globally defined differentials on the spectral curve rather than formal power series expansions, satisfy the so-called \emph{quadratic loop equations}. For a collection of symmetric differentials $ (\omega_{g,n})_{g\geq 0, n \geq 1}$ on a spectral curve, the quadratic loop equations state that for all $ g \geq 0$ and $n \geq 1$
	\begin{equation}
	\label{QLE}
	\omega_{g-1,n+1}(z, \sigma_i(z), z_{\llbracket n-1 \rrbracket} ) + \sum_{\substack{g = g_1 + g_2\\ \llbracket n-1 \rrbracket = I \sqcup J}} \omega_{g_1,|I|+1}(z, z_I) \omega_{g_2,|J|+1}(\sigma_i(z),z_J)
	\end{equation}
	is holomorphic in $z$ near $p_i$, with a double zero at $ p_i$ itself, cf. \cite[(2.2)]{BorotShadrin}.
\end{enumerate}
The relation between \cite[Theorem 2.2]{BorotShadrin} and the list above is given by lemma~\ref{lem:lle}.\par
As we mentioned above, the unstable cases are proved in~\cite{DoKarev,KLS}, and in an unpublished work of Karev. The second property is proved in~\cite{KLS}. So, the only thing that we have to do to complete the proof is to formulate and prove the quadratic loop equations. It is done in proposition~\ref{prop:QLE} below.\hfill $\Box$

\begin{remark}
	This approach to proving the topological recursion was used before in~\cite{DLPS,DBKraPS} (where the quadratic loop equations followed directly from the cut-and-join equation) and in~\cite{BKLPS,r-spinFullProof}, where a system of formal corollaries of the quadratic loop equations was related to the cut-and-join operators of completed $r$-cycles. In this paper we combine the latter result with the formula in~\cite[Example 5.8]{ALS} that expresses the partition function of the monotone orbifold Hurwitz numbers in terms of an infinite series of the operators of completed $r$-cycles.
\end{remark}

\subsection{Organization of the paper} This paper is very essentially based on the results of~\cite{r-spinFullProof} and~\cite{ALS}. However, in this paper, we work exclusively in the so-called bosonic Fock space, i.e. the space of symmetric functions instead of the fermionic Fock space, or semi-infinite wedge formalism, as in \emph{op.~cit.}. By the classical boson-fermion correspondence~\cite{KacBook,MiwaJimboDate}, we can translate the necessary results in the fermionic Fock space to the language of differential operators in the ring of symmetric functions.

In section~\ref{sec:CJ} we derive the so-called ``cut-and-join'' evolutionary equation for the exponential partition function of monotone orbifold Hurwitz numbers and discuss its convergence issues. In section~\ref{sec:Holom} we use the cut-and-join operator to construct a particular expression holomorphic at the critical points of the spectral curve, which is needed for the proof of the quadratic loop equations. In section~\ref{sec:QLE} we formulate and prove the quadratic loop equations.

\subsection{Acknowledgments} We thank A.~Alexandrov, P.~Dunin-Barkowski, M.~Karev, and D.~Lewa\'nski for useful discussions. We also thank an anonymous referee for useful suggestions.
A.P. would like to thank Korteweg-de Vries Institute for hospitality and flourishing scientific atmosphere.
R.K. and S.S. were supported partially by the Netherlands Organization for Scientific Research.
A.P. was supported in part by the grant of the Foundation for the Advancement of Theoretical Physics “BASIS",
 by RFBR grants 16-01-00291, 18-31-20046 mol\_a\_ved and 19-01-00680 A.

\section{The cut-and-join operator}
\label{sec:CJ}

Define the function $ \zeta (z) = e^{z/2} - e^{-z/2}$ and for a partition $ \lambda$ (viewed as its Young diagram), and a box $ \square = (i,j) \in \lambda$, let $ \mathsf{cr}^\lambda_\square = i-j$ be its content.  The partition function of the monotone $q$-orbifold Hurwitz numbers can be defined as \cite{HarnadOrlov}
\begin{equation}\label{partitionfunction}
Z\coloneqq \sum\limits_{g=0}^\infty\sum\limits_{\mu} \frac{\hbar^{2g-2+l(\mu)+|\mu|/q}}{|\mathrm{Aut}(\mu)|}h^\bullet_{g,\mu}\prod\limits_{i=1}^{l(\mu)}p_{\mu_i}= \sum_\lambda s_\lambda (\delta_q ) \Big( \prod_{\square \in \lambda} (1-\hbar \mathsf{cr}^\lambda_\square )^{-1}\Big) s_\lambda (p)\,,
\end{equation}
where the $s_\lambda$ are Schur functions expressed as polynomials in the power sums $ p_i$, and the left Schur function is evaluated at the point $ p_j = \delta_{j,q}$.

Define the series of operators $Q(z)=\sum_{r=1}^\infty Q_rz^r$ as 
\begin{equation}
Q(z)\coloneqq\frac{1}{\zeta(z)} \sum_{s=1}^\infty \left(
\sum_{\substack{
n\geq 1 \\
k_1,\dots,k_n\geq 1 \\
k_1+\cdots+k_n=s
}} 
\frac 1{n!}
\prod_{i=1}^n
\frac{\zeta(k_i z) p_{k_i}}{k_i}
\right)
\left(
\sum_{\substack{
		m\geq 1 \\
		\ell_1,\dots,\ell_m\geq 1 \\
		\ell_1+\cdots+\ell_m=s
}} 
\frac 1{m!}
\prod_{j=1}^m
\zeta(\ell_j z) \frac{\partial}{\partial p_{\ell_j}}
\right).
\end{equation}
Define the operator $J$ as 
\begin{align}
J &\coloneqq  \frac{\frac{\partial}{\partial \hbar}}{\zeta\left(\hbar^2 \frac{\partial}{\partial \hbar}\right)} \sum_{r=1}^\infty \hbar^r Q_r (r-1)! 
-
\frac{1}{\hbar} Q_1
\\ \notag
& = \sum_{r=2}^\infty \hbar^{r-2} Q_r (r-1)! +\sum_{\alpha=1}^\infty c_\alpha \sum_{r=1}^\infty \hbar^{r-2+2\alpha} Q_r (r-1+2\alpha)!.
\end{align}
Here $c_\alpha$ are the coefficients of the expansion $\frac{z}{\zeta(z)}= \sum_{\alpha=0}^\infty c_\alpha z^{2\alpha}$, that is, $c_1=-\frac{1}{24}, c_2=\frac 7{5760}$, and in general $c_\alpha$ can be expressed in terms of the Bernoulli numbers as $c_\alpha = \frac{2^{1-2\alpha} B_{2\alpha}}{(2\alpha)!}$.
\begin{proposition} 
\label{DifEqPartFun}
We have: 
$
\frac{\partial }{\partial \hbar} Z= JZ.
$
\end{proposition}

\begin{proof} Recall~\cite[proposition 5.2]{ALS}, which states that the operator $\mathcal{D}(\hbar)$ acting on the space of symmetric functions as $\mc{D} (\hbar)s_\lambda \coloneqq \left[\prod_{\square \in \lambda} (1-\hbar \mathsf{cr}^\lambda_\square )^{-1}\right]s_\lambda$ as in equation~\eqref{partitionfunction} can be expressed by the formula
\begin{equation}
\mathcal{D}(\hbar)=\exp\left(\left[
\tilde{\mathcal{E}}_0(\hbar^2 \frac{\partial}{\partial \hbar}) / \zeta(\hbar^2 \frac{\partial}{\partial \hbar}) 
-\mathcal{F}_1
\right]\log(\hbar) \right) \,,
\end{equation}
where $\tilde{\mathcal{E}}_0(z) = \sum_{r=1}^\infty \mathcal{F}_r\frac{z^r}{r!}= z \sum_{r=1}^\infty \mathcal{F}_{r}\frac{z^{r-1}}{r!}$, and $\mathcal{F}_r$ is the operator whose action in the basis of Schur polynomials is diagonal and is given by
\begin{equation}
\mathcal{F}_{r} s_\lambda = \sum_{i=1}^\ell \left( (\lambda_i -i +\frac 12)^r - (-i+\frac 12)^r\right)s_\lambda
\end{equation}
for $\lambda = (\lambda_1\geq \lambda_2\geq \cdots \geq \lambda_\ell)$ \cite[equation (2.4)]{ALS}. The operators $\mathcal{F}_r$ can be expressed as differential operators in the variables $p$ as  $\mathcal{F}_r s_\lambda = r! Q_r s_\lambda $, $r\geq 1$ (\cite[theorem 5.2]{SSZ-LMS}, see also~\cite{Alexandrov-CJ,Rossi-CJ}).
Note that 
\begin{align}
\frac{\partial}{\partial \hbar}\mathcal{D}(\hbar)
&=\frac{1}{\hbar^2}\cdot \hbar^2\frac{\partial}{\partial \hbar}\mathcal{D}(\hbar) = \mathcal{D}(\hbar) \cdot \frac{1}{\hbar^2}\left(\left[
\tilde{\mathcal{E}}_0(\hbar^2 \frac{\partial}{\partial \hbar}) / \zeta(\hbar^2 \frac{\partial}{\partial \hbar}) - \mathcal{F}_1
\right]\hbar \right)\\ \notag
&=\mathcal{D}(\hbar) \cdot \left(\frac{1}{\hbar^2} \left(
\frac{\hbar^2 \frac{\partial}{\partial \hbar}}{\zeta(\hbar^2 \frac{\partial}{\partial \hbar})}
\sum_{r=1}^\infty \mathcal{F}_{r}\frac{\hbar^{r}}{r}
\right)
-\frac{1}{\hbar}\mathcal{F}_1 \right)
\end{align}
and, therefore, 
\begin{align}
\frac{\partial}{\partial \hbar}Z & = \sum_\lambda s_\lambda( \delta_q) \mathcal{D}(\hbar)  \left[
\frac{\frac{\partial}{\partial \hbar}}{\zeta(\hbar^2 \frac{\partial}{\partial \hbar})}
\sum_{r=1}^\infty \mathcal{F}_{r}\frac{\hbar^{r}}{r}
-\frac{1}{\hbar}\mathcal{F}_1
\right]
s_\lambda (p)
\\ \notag
& = \sum_\lambda s_\lambda( \delta_q) \mathcal{D}(\hbar)  \left[
\frac{\frac{\partial}{\partial \hbar}}{\zeta(\hbar^2 \frac{\partial}{\partial \hbar})}
\sum_{r=1}^\infty \hbar^r Q_{r} (r-1)!
-\frac{1}{\hbar}Q_1
\right]
s_\lambda (p)
=JZ.
\end{align}
\end{proof}

\begin{corollary}\label{cor:CJ}
 For $2g-2+n>0$ we have:
\begin{align} \label{eq:CutAndJoinNPoint}
& 
\bigg(2g-2+n+\frac{1}{q}\sum_{i=1}^n D_{x_i}\bigg) \tilde H_{g,n} =
\\ \notag
& 
\sum_{\substack{m \geq 1, d \geq 0 \\ m + 2d \geq 2}} 
\frac{(m+2d-1)!}{m!} 
\sum_{\ell =1}^m 
\frac{1}{\ell!}
\sum_{\substack{
\{ k\} \sqcup \bigsqcup_{j=1}^\ell K_j = \llbracket n \rrbracket \\ 
\bigsqcup_{j=1}^\ell M_j =  \llbracket m \rrbracket\\ 
M_j \neq \emptyset \\
g-d = \sum_{j=1}^\ell g_j + m - \ell \\
g_1,\ldots,g_{\ell} \geq 0 
}}  
Q_{d,\emptyset,m}^{(k)} \bigg[\prod_{j = 1}^{\ell} \tilde{H}_{g_j,|M_j|+|K_j|}(\xi_{M_j},x_{K_j})\bigg]
\\ \notag
&
+
\sum_{\alpha=1}^g c_\alpha
\sum_{\substack{m \geq 1, d \geq 0 \\ m + 2d \geq 1}} 
\frac{(m+2d-1+2\alpha)!}{m!} 
\sum_{\ell =1}^m 
\frac{1}{\ell!}
\sum_{\substack{
		\{ k\} \sqcup \bigsqcup_{j=1}^\ell K_j = \llbracket n \rrbracket \\ 
		\bigsqcup_{j=1}^\ell M_j =  \llbracket m \rrbracket\\ 
		M_j \neq \emptyset \\
		g-d-\alpha = \sum_{j=1}^\ell g_j + m - \ell \\
		g_1,\ldots,g_{\ell} \geq 0 
}}  
Q_{d,\emptyset,m}^{(k)} \bigg[\prod_{j = 1}^{\ell} \tilde{H}_{g_j,|M_j|+|K_j|}(\xi_{M_j},x_{K_j})\bigg],
\end{align}
where $D_{x_i} = x_i \frac{\partial}{\partial x_i}$,
\begin{align}
\label{Qrdef} 
 \sum_{d \geq 0} Q_{d;K_0 ,m}^{(k)}\,z^{2d} & = \frac{z}{\zeta(z)} \prod_{i \in \{k\} \sqcup K_0} \frac{\zeta(zD_{x_i})}{zD_{x_i}} \circ \prod_{j = 1}^m \frac{\zeta(zD_{\xi_j})}{z}\bigg|_{\xi_j = x_k}, & D_{\xi_j} &= \xi_j \frac{\partial}{\partial \xi_j}\,,
\\
\tilde H_{0,1}(\xi) & = H_{0,1}(\xi) 
\\
\tilde{H}_{0,2}(\xi,x) & = H_{0,2}(\xi,x) + H^{\textup{sing}}_{0,2}(\xi, x) \,, & H^{\textup{sing}}_{0,2}(\xi, x)& =\log \Big(\frac{\xi - x}{\xi x}\Big),\label{TildeH02}
\\
\tilde{H}_{0,2}(\xi_1,\xi_2) & = H_{0,2}(\xi_1,\xi_2),
\\
\tilde{H}_{g,n} & = H_{g,n} + \sum_{\alpha = 0}^g c_\alpha \frac{(2g-2+n+2\alpha )!}{2g-2+n}\,, & 2g-2+n&>0.\label{TildeHgn}
\end{align}
\end{corollary}

The contribution $H^{\textup{sing}}_{0,2}(\xi, x)$ is called the \emph{singular part}. Note that we introduce more general operators $Q_{d;K_0 ,m}^{(k)}$ than the ones used in the statement of the corollary (where only have $K_0=\emptyset$), since we need them below in the proof.

\begin{proof} The proof repeats \emph{mutatis mutandis} the proof of~\cite[proposition 10]{BKLPS}, so we only give a sketch of the idea, with the analogy explained. The operator $J$ is a linear combination of the $ Q_r$. Hence, comparing proposition~\ref{DifEqPartFun} to \cite[equation (3)]{BKLPS}: $ \frac{1}{r!} \frac{\partial}{\partial \beta} Z^{r,q} = Q_{r+1}Z^{r,q}$,  we can manipulate the first equation as the second. So, we map $ p_\mu$ to monomial symmetric functions $ \textup{M}_\mu(x_1, \dotsc, x_n)$, using \cite[equation~(5)]{BKLPS} for the effect of this map on the operators $Q_r$ acting on a partition function $Z$. The next step is incorporating the factors $ \frac{x_i}{x_k-x_i}$ as part of $ \tilde{H}_{0,2}$, which is given by \eqref{TildeH02} and explained in the proof of \cite[proposition~10]{BKLPS}. As in that proof, this adds a term on the right-hand side of equation~\eqref{eq:CutAndJoinNPoint} where all factors are singular parts, and this is the extra term in \eqref{TildeHgn}. This corresponds to the case $m=\ell = n-1$, and can equivalently be written in the shape of~\cite[proposition 6]{BKLPS}, with $m=\ell = 0$. This gives
\begin{equation}
\begin{split}
\sum_{d \geq \min \{ 0,3-n\}} (n+2d-2)! \sum_{\{k\} \sqcup K_0 = \llbracket n\rrbracket} \delta_{g,d} Q^{(k)}_{d,K_0,0} \prod_{j \in K_0} \frac{x_j}{x_k-x_j} \\
+\sum_{\alpha =1}^g c_\alpha (n+2d+2\alpha -2)! \sum_{\{k\} \sqcup K_0 = \llbracket n\rrbracket} \delta_{g,d} Q^{(k)}_{d,K_0,0} \prod_{j \in K_0} \frac{x_j}{x_k-x_j} \,.
\end{split}
\end{equation}
The condition $ d \geq \min \{ 0,3-n\}$ excludes the unstable cases $ (g,n) = (0,1), (0,2)$, and for $2g-2+n >0$  it simplifies to
\begin{align}
\sum_{\alpha =0}^g &c_\alpha(n+2g+2\alpha -2)! \sum_{\{k\} \sqcup K_0 = \llbracket n\rrbracket} \prod_{i =1}^n \frac{\zeta(zD_{x_i})}{zD_{x_i}} \prod_{j \in K_0} \frac{x_j}{x_k-x_j} \\
&= \sum_{\alpha =0}^g c_\alpha (n+2g+2\alpha -2)! \prod_{i =1}^n \frac{\zeta(zD_{x_i})}{zD_{x_i}}  \sum_{k=1}^n \prod_{\substack{j=1\\j\neq k}}^n \frac{x_j}{x_k-x_j} \,.
\end{align}
As calculated in~\cite[proposition 10]{BKLPS}, 
\begin{equation}
\sum_{k=1}^n \prod_{\substack{j=1\\j\neq k}}^n \frac{x_j}{x_k-x_j} =-1
\end{equation}
and therefore, there cannot be any derivatives acting on it.\par
Now we use induction on $ 2g-2+n$, with the induction hypothesis being that $ H_{g,n} - \tilde{H}_{g,n} $ is a constant. This holds for the $(0,1)$ case, while the $(0,2)$ case is taken care of by the previous argument. Using the induction hypothesis, we get from the previous calculation that
\begin{equation}
\bigg(2g-2+n+\frac{1}{q}\sum_{i=1}^n D_{x_i}\bigg)(H_{g,n}- \tilde H_{g,n}) = \sum_{\alpha =0}^g c_\alpha (n+2g+2\alpha -2)!\,,
\end{equation}
as all constants from previous $H-\tilde{H}$ are annihilated on the right-hand side by derivatives.\par
As both $H_{g,n} $ and $ \tilde{H}_{g,n}$ are power series in the $x_i$ and the $D_{x_i}$ preserve degree and vanish on constants, this shows that
\begin{equation}
H_{g,n}- \tilde H_{g,n} = \sum_{\alpha =0}^g c_\alpha \frac{(n+2g+2\alpha -2)!}{2g-2+n}\,.
\end{equation}
\end{proof}

\begin{remark}\label{rem:globalfunction} It is proved in~\cite{KLS} that each $H_{g,n}$ is an expansion of a globally defined meromorphic function on $\mathbb{C}^n$ with known positions of poles and bounds on their order. More precisely, for $2g-2+n>0$, $H_{g,n}(x_ {\llbracket n \rrbracket})$ is the expansion of a function of  $z_{\llbracket n \rrbracket}$, $x_i=x(z_i)$, which, by a slight abuse of notation, we also denote by $H_{g.n}(z_ {\llbracket n \rrbracket})$, with the poles in each variable only at the points $p_1,\dots,p_q$, where the order of poles is bounded by some constants that depend only on $g$ and $n$.
\end{remark}

Remark~\ref{rem:globalfunction} implies, in particular, that the right hand side of equation~\eqref{eq:CutAndJoinNPoint} is an infinite sum of meromorphic functions on $\mathbb{C}^{n}$ with the natural coordinates $z_1,\dots,z_n$, $x_i=x(z_i)$, with the poles in each variable only at the points $p_1,\dots,p_q$ and on the diagonals, where the order of poles is bounded by some constants that depend only on $g$ and $n$. Let us prove that this infinite sum converges absolutely and uniformly on every compact subset of $(D\setminus \{p_1,\dots,p_q\})^{n}\setminus\mathrm{Diag}$ to a meromorphic function with the same restriction on poles (and, therefore, equation~\eqref{eq:CutAndJoinNPoint} makes sense). Here, $D$ is the unit disc.

\begin{lemma}\label{lem:convergence} 
Corollary~\ref{cor:CJ} holds on the level of meromorphic functions on the unit disc $D$ in the variables $z_i$, $i=1,\dots,n$: the right hand side converges absolutely and uniformly on every compact subset of $(D\setminus \{p_1,\dots,p_q\})^{n}\setminus\mathrm{Diag}$ to a meromorphic function with the poles in each variable only at the points $p_1,\dots,p_q$ and on the big diagonal, the locus where at least two coordinates are equal. The order of poles is bounded by some constants that depend only on $g$ and $n$.
\end{lemma}
\begin{proof}
In order to see the convergence, we have to rewrite each of the summands on the right hand side (the first summand and the coefficients of $c_\alpha$) in a way that collects all but finitely many terms in a series that can be analysed well. We claim that the only source of infinite summation are factors $D_\xi \tilde H_{0,1}(\xi)$. To see this, let us first analyse the summation range of \eqref{eq:CutAndJoinNPoint} for a given $(g,n)$. Let us work it out for the first summand, the computation for all other summands is exactly the same. One summation condition is $g-d = \sum_{j=1}^\ell g_j + m - \ell$, which can be rewritten as $g = d+  \sum_{j=1}^\ell (g_j + |M_j| - 1)$. As $ g_j + |M_j| -1 >0$ unless $ (g,|M_j| ) = (0,1)$, and furthermore there are only finitely many $x_i$ to distribute, this does show that the sum over $m$, $d$ (which bounds the number of $ D$), decompositions of $ \llbracket n \rrbracket $, and $ g_j$ is finite if we exclude $ D_{x_k} \tilde{H}_{0,1}$. Furthermore, each such term obtains an infinite `tail' of $D_{x_k} \tilde{H}_{0,1}$, as follows, where the variable $ m$ on the first line is split into $ m$ and $ t$ on the second and third line: 
\begin{align}
& 
\sum_{\substack{m \geq 1, d \geq 0 \\ m + 2d \geq 2}} 
\frac{(m+2d-1)!}{m!} 
\sum_{\ell =1}^m 
\frac{1}{\ell!}
\sum_{\substack{
		\{ k\} \sqcup \bigsqcup_{j=1}^\ell K_j = \llbracket n \rrbracket \\ 
		\bigsqcup_{j=1}^\ell M_j =  \llbracket m \rrbracket\\ 
		M_j \neq \emptyset \\
		g-d = \sum_{j=1}^\ell g_j + m - \ell \\
		g_1,\ldots,g_{\ell} \geq 0 
}}  
Q_{d,\emptyset,m}^{(k)} \bigg[\prod_{j = 1}^{\ell} \tilde{H}_{g_j,|M_j|+|K_j|}(\xi_{M_j},x_{K_j})\bigg] =
\\ \notag
& \sum_{m, d \geq 0} 
\frac{1}{m!} 
\sum_{\ell =0}^m 
\frac{1}{\ell!}
\sum_{\substack{
		\{ k\} \sqcup \bigsqcup_{j=1}^\ell K_j = \llbracket n \rrbracket \\ 
		\bigsqcup_{j=1}^\ell M_j =  \llbracket m \rrbracket\\ 
		M_j \neq \emptyset \\
		g-d = \sum_{j=1}^\ell g_j + m - \ell \\
		g_1,\ldots,g_{\ell} \geq 0 
}}  
\left[Q_{d,\emptyset,m}^{(k)} \bigg[\prod_{j = 1}^{\ell} \tilde{H}_{g_j,|M_j|+|K_j|}(\xi_{M_j},x_{K_j})\bigg] \right]_{\text{no}\, D_{x_k} \tilde H_{0,1}(x_k)}
\\ \notag
&
\phantom{
\sum_{\substack{m \geq 1, d \geq 0 \\ m + 2d \geq 2}} 
	\frac{1}{m!} 
	\sum_{\ell =1}^m 
	\frac{1}{\ell!}
	\sum_{
			g-d = \sum_{j=1}^\ell g_j + m - \ell }  
}
\times
\sum_{\substack{t=0\\t+m \geq 1\\t+m+2d \geq 2}}^\infty \frac{(m+t+2d-1)!}{t!} \left(D_{x_k} \tilde H_{0,1}(x_k)\right)^t.
\end{align}
Here we mean that in the second line, we exclude any factors $D_{x_k} \tilde{H}_{0,1}(x_k)$ left after the action of $Q$, and we collect these in the third line.
Now the first two summations are finite, the coefficients 
\begin{equation}
\left[Q_{d,\emptyset,m}^{(k)} \bigg[\prod_{j = 1}^{\ell} \tilde{H}_{g_j,|M_j|+|K_j|}(\xi_{M_j},x_{K_j})\bigg] \right]_{\text{no}\, D_\xi \tilde H_{0,1}(\xi)}
\end{equation}
are meromorphic functions with the desired restriction on poles, and the sum over $t$ determines the explicit functions $ \sum_{t=0}^\infty \frac{(m+t+2d-1)!}{t!} u^t = (\frac{d}{du})^{m+2d-1} \frac{u^{m+2d-1}}{1-u}$ of its argument $u = z^q =xy = D_{x_k} \tilde H_{0,1}(x_k)$, which converge on the unit disc.
\end{proof}

\section{Holomorphic expression}
\label{sec:Holom}

In this section we analyse a symmetrization of equation~\eqref{eq:CutAndJoinNPoint} near one of the critical points of the function $x(z)$. For the rest of this paper we fix $p=p_j$, $j=1,\dots, q$, and by $z\mapsto \bar{z}$ we denote the deck transformation near $p$. 

We define the \emph{symmetrizing operator} $\cS_z$ and the \emph{anti-symmetrizing operator} $\Delta_z$ by
\begin{align}
	\cS_zf(z) &\coloneqq f(z) + f(\bar{z})\,;\\ \notag
	\Delta_zf(z) &\coloneqq f(z) - f(\bar{z})\,,
\end{align}
and use the identity~\cite{BKLPS,r-spinFullProof}
\begin{equation}\label{Sondiagonal}
\cS_z \Big(f(z_1,\dots,z_r) \Big|_{z_i = z} \Big) = 2^{1-r} \Big( \sum_{\substack{I \sqcup J = \llbracket r \rrbracket \\ |J|\,\,\text{even}}} \Big( \prod_{i \in I} \cS_{z_i} \Big) \Big( \prod_{j \in J} \Delta_{z_j} \Big)f(z_1,\dotsc, z_r) \Big)\Big|_{z_i = z}\,, \qquad r\geq 1.
\end{equation}

Recall remark~\ref{rem:globalfunction}. Another direct corollary of the results of~\cite{KLS} is the linear loop equations for the $n$-point functions that can be formulated as the following lemma:

\begin{lemma}
\label{lem:lle} 
For any $g\geq0$ and $n\geq 1$ we have: $\cS_{z_i} H_{g,n}(z_{\llbracket n \rrbracket})$ is holomorphic at $z_i\to p$. 
\end{lemma}

\begin{proof}
By \cite[theorem~5.2~\&~proposition~6.2]{KLS}, cf. the statement after the proof of that proposition, the $H_{g,n}$ are linear combinations of polynomials in $ \big\{ \frac{d}{d x_i} \big\} $ acting on $ \prod_{i=1}^n \xi_{\alpha_i} (x_i)$, where, up to a linear change of basis given in \cite[Section~6.1]{KLS}, $ \xi_\alpha (z) = \frac{1}{z - p_\alpha}$. In particular, $ \cS_{z_i} \xi_\alpha (z_i) $ is holomorphic at $ z_i \to p$ (trivially if $ p \neq p_\alpha$ and because the pole is odd if $ p = p_\alpha$). Because $x$ itself is invariant under the involution by definition, this holomorphicity is preserved under any amount of applications of $ \frac{d}{d x_i}$.
\end{proof}

In order to simplify the notation, consider equation~\eqref{eq:CutAndJoinNPoint} for $H_{g,n+1}=H_{g,n+1}(x_0,\dots,x_n)$, and substitute $ x_i = x(z_i)$. We apply the operator $\cS_{z_0}$ to both sides of this equation. From the linear loop equations we immediately see that the left hand side of this equation is holomorphic at $z_0\to p$, as well as all summands on the right hand of this equation with $k\neq 0$ (it is an infinite sum that converges in the sense of lemma~\ref{lem:convergence}). Thus we know that  
\begin{lemma} \label{lem:holomorphicity-1} The expression
\begin{align} \label{eq:holomorphicexpression}
& \cS_{z_0} \Bigg[\sum_{\substack{m \geq 1, d \geq 0 \\ m + 2d \geq 2}} 
\frac{(m+2d-1)!}{m!} 
\sum_{\ell =1}^m 
\frac{1}{\ell!}
\!\!\!\!\!
\sum_{\substack{
		\bigsqcup_{j=1}^\ell K_j = \llbracket n \rrbracket \\ 
		\bigsqcup_{j=1}^\ell M_j =  \llbracket m \rrbracket\\ 
		M_j \neq \emptyset \\
		g-d = \sum_{j=1}^\ell g_j + m - \ell \\
		g_1,\ldots,g_{\ell} \geq 0 
}}  \!\!\!\!\!
Q_{d,\emptyset,m}^{(0)} \bigg[\prod_{j = 1}^{\ell} \tilde{H}_{g_j,|M_j|+|K_j|}(\xi_{M_j},x_{K_j})\bigg]
\\ \notag
&
 +
\sum_{\alpha=1}^g c_\alpha
\sum_{\substack{m \geq 1, d \geq 0 \\ m + 2d \geq 1}} 
\frac{(m+2d-1+2\alpha)!}{m!} 
\sum_{\ell =1}^m 
\frac{1}{\ell!}
\!\!\!\!\!\!\!\sum_{\substack{
		\bigsqcup_{j=1}^\ell K_j = \llbracket n \rrbracket \\ 
		\bigsqcup_{j=1}^\ell M_j =  \llbracket m \rrbracket\\ 
		M_j \neq \emptyset \\
		g-d-\alpha = \sum_{j=1}^\ell g_j + m - \ell \\
		g_1,\ldots,g_{\ell} \geq 0 
}}  \!\!\!\!\!\!\!
Q_{d,\emptyset,m}^{(0)} \bigg[\prod_{j = 1}^{\ell} \tilde{H}_{g_j,|M_j|+|K_j|}(\xi_{M_j},x_{K_j})\bigg]
\Bigg]
\end{align}
is holomorphic at $z_0\to p$. 
\end{lemma}

It is convenient to introduce the notation 
\begin{align}
& W_{g,n+m} (\xi_{\llbracket m \rrbracket} , x_{\llbracket n \rrbracket})\coloneqq D_{\xi_1}\cdots D_{\xi_m}D_{x_1}\cdots D_{x_n} \tilde{H}_{g,m+n}(\xi_{\llbracket m \rrbracket},x_{\llbracket n \rrbracket}); \\
& \cW_{g,m,n}(\xi_{\llbracket m \rrbracket} \mid z_{\llbracket n \rrbracket}) \coloneqq 
\sum_{\ell =1}^m 
\frac{1}{\ell!}
\!\!
\sum_{\substack{
		\bigsqcup_{j=1}^\ell K_j = \llbracket n \rrbracket \\ 
		\bigsqcup_{j=1}^\ell M_j =  \llbracket m \rrbracket\\ 
		M_j \neq \emptyset \\
		g= \sum_{j=1}^\ell g_j + m - \ell \\
		g_1,\ldots,g_{\ell} \geq 0 
}}  \!\!
\prod_{j = 1}^{\ell} W_{g_j,|M_j|+|K_j|}(\xi_{M_j},x_{K_j}),
\end{align}
where we assume that $\xi_i\coloneqq x(w_i)$, $i=1,\dots,m$, and $x_j\coloneqq x(z_j)$, $j=1,\dots,n$. Denote also
\begin{align}
\sum_{d \geq 0} \cQ_{d,m}(z_0)t^{2d} & = \frac{t}{\zeta(t)} \frac{\zeta(tD_{x(z_0)})}{tD_{x(z_0)}} \circ 
\prod_{j = 1}^m \bigg( \left[\vert_{w_j = z_0}\right] \circ \frac{\zeta(tD_{x(w_j)})}{tD_{x(w_j)}}\bigg),
\end{align}
where
\begin{align}
\left[\vert_{w_j = z_0}\right] F(w) &\coloneq \Res_{w=z} F(w)\frac{dx(w)}{x(w)-x(z)}
\,.
\end{align}
These notations allow us to rewrite expression~\eqref{eq:holomorphicexpression} and to reformulate lemma~\ref{lem:holomorphicity-1} as
\begin{corollary}\label{cor:holomorphicinput} The expression
\begin{align}\label{eq:holoinput}
&\cS_{z_0} \left[\sum_{\substack{m \geq 1, d \geq 0 \\ m + 2d \geq 2}} 
\frac{(m+2d-1)!}{m!} \cQ_{d,m}(z_0) \cW_{g-d,m,n} (w_{\llbracket m \rrbracket} \mid z_{\llbracket n \rrbracket}) \right.
\\ \notag & 
\left.
+
\sum_{\alpha=1}^g c_\alpha
\sum_{\substack{m \geq 1, d \geq 0 \\ m + 2d \geq 1}} 
\frac{(m+2d-1+2\alpha)!}{m!} \cQ_{d,m}(z_0) \cW_{g-d-\alpha,m,n} (w_{\llbracket m \rrbracket} \mid z_{\llbracket n \rrbracket})  \right].
\end{align}
	is holomorphic at $z_0\to p$. 
\end{corollary}

\section{Quadratic loop equations} 
\label{sec:QLE}
In this section we use corollary~\ref{cor:holomorphicinput} and results of~\cite{r-spinFullProof} for the proof of the quadratic loop equations that can be formulated as the following proposition:
\begin{proposition}\label{prop:QLE} For any $g\geq 0$, $n\geq 0$ we have: $\cW_{g,2,n}(w,\bar{w}\mid z_{\llbracket n \rrbracket})$ is holomorphic at $w\to p$. 
\end{proposition}

\begin{remark}
In order to see that this really gives the quadratic loop equation, as given in \eqref{QLE}, note that
\begin{equation}
\cW_{g,2,n}(w,\bar{w} \mid x_{\llbracket n\rrbracket}) = W_{g-1,n+2}(w, \bar{w}, z_{\llbracket n \rrbracket}) + \sum_{\substack{K_1 \sqcup K_2 = \llbracket n \rrbracket \\ g = g_1 + g_1}} W_{g_1,1+|K_1|}(w,z_{K_1}) W_{g_2,1+|K_2|}(\bar{w},z_{K_2})\,.
\end{equation}
Furthermore, $ d_1 \otimes \dotsb \otimes d_n H_{g,n}(z_{\llbracket n\rrbracket} ) = W_{g,n}(z_{\llbracket n \rrbracket} ) \prod_{i=1}^n \frac{dx(z_i)}{x(z_i)}$. Therefore, $\cW_{g,2,n}(w,\bar{w}\mid z_{\llbracket n \rrbracket})$ is holomorphic at $ w \to p$ if and only if \eqref{QLE} is holomorphic with double zero there.
\end{remark}

Let us explain the strategy of the proof. We prove this proposition by induction on the negative Euler characteristic, that is, on $2g-2+(n+1)$. We split the known to be holomorphic at $z_0\to p$ expression~\eqref{eq:holoinput}, which is an infinite sum of meromorphic functions converging in the sense of lemma~\ref{lem:convergence}, into a sum of two converging infinite sums, where one sum is holomorphic once the quadratic loop equations hold for all $(g',n')$ with $2g'-2+(n'+1)<2g-2+(n+1)$, and the other sum is holomorphic if and only if the quadratic loop equation holds for $(g,n)$.

To this end, we have to recall some of the results of~\cite{r-spinFullProof}. First of all, we need a change of notation in the case when we apply $\Delta_{w_i}\Delta_{w_j}$ and $\cS_{w_i}\cS_{w_j}$ operators to $W_{0,2}(\xi_i,\xi_j))$, $\xi_i=x(w_i)$, $\xi_j=x(w_j)$ (which is a possible factor in $\cW$)---see~\cite[section 3.1]{r-spinFullProof} for a motivation of this change of notation. So, we redefine
\begin{align}\label{eq:redefinition02}
\widetilde{\Delta_{w_i}\Delta_{w_j} } W_{0,2}(\xi_i,\xi_j) & \coloneqq \Delta_{w_i}\Delta_{w_j} W_{0,2}(\xi_i,\xi_j)-\frac{2}{(\log\xi_i-\log\xi_j)^2}\,;
\\ \notag 
\widetilde{\cS_{w_i}\cS_{w_j} } W_{0,2}(\xi_i,\xi_j) & \coloneqq \cS_{w_i}\cS_{w_j} W_{0,2}(\xi_i,\xi_j) +\frac{2}{(\log\xi_i-\log\xi_j)^2}\,.
\end{align}
From now on, we use this modified definition, and abusing notation we always omit the tildes.

Recall that all $H_{g,n}$'s satisfy the linear loop equations (lemma~\ref{lem:lle}). Under the assumption that the quadratic loop equations hold for all $(g',n')$ with $2g'-2+(n'+1)<2g-2+(n+1)$  the following two lemmas hold: 

\begin{lemma} \label{lem:corollaryDKPS}
For any $r \geq 0$ and any $h, k \geq 0$ such that $ 2h-1+k -r\leq 2g-2+n$, the expression
\begin{equation}\label{eq:corollaryDKPS}
\sum_{m=1}^{r+1}\frac{1}{m!} \sum_{\substack{2\alpha_1 + \dotsb + 2\alpha_m \\ + m = r+1}} \prod_{j=1}^m \bigg( \left[\vert_{w_j = z_0}\right] \frac{D_{x(w_j)}^{2\alpha_j}}{(2\alpha_j+1)!} \bigg) \sum_{\substack{I\sqcup J = \llbracket m \rrbracket \\ |I|\in 2\mathbb{Z} }} \prod_{i\in I} \Delta_{w_i} \prod_{j\in J} \cS_{w_j} \cW_{h- \alpha_1 - \dotsc - \alpha_m,m,k} (w_{\llbracket m\rrbracket} \mid z_{\llbracket k\rrbracket}) 
\end{equation}
as well as its arbitrary $D_{x(z_0)}$-derivatives, is holomorphic at $z_0\to p$.
\end{lemma}
\begin{proof} This is a direct corollary of~\cite[corollary 3.4]{r-spinFullProof}.
\end{proof}

\begin{lemma}\label{lem:secondRSpinLemma} For any $r\geq 1$
\begin{align}\label{eq:expr-r}
& \sum_{\substack{k,\alpha_1,\dots,\alpha_{2k}\\ \ell, \beta_1,\dots,\beta_\ell \\ 2k+2\alpha_1+\cdots+2\alpha_{2k} \\ +\ell + 2\beta_1+\cdots+2\beta_{\ell} = r+1}}\!\!\!\!\!\!\!\!\! \frac{1}{\ell! (2k)!}\prod_{i=1}^\ell \left[\vert_{w'_i = z_0}\right] \frac{D_{x(w'_i)}^{2\beta_i}}{(2\beta_i+1)!} \cS_{w'_i} 
\prod_{i=1}^{2k} \left[\vert_{w_i = z_0}\right] \frac{D_{x(w_i)}^{2\alpha_i}}{(2\alpha_i+1)!} \Delta_{w_i} \cW_{g+(2k+\ell-r-1)/2,\ell+2k,n} (w'_{\llbracket \ell\rrbracket},w_{\llbracket 2k\rrbracket} \mid z_{\llbracket n\rrbracket}) 
\\ \notag &
-\sum_{2k+\ell = r+1} \frac{1}{\ell! (2k)!} \binom{k}{1} \left(\cS_{z_0} W_{0,1}(x(z_0))\right)^\ell \left(\Delta_{z_0} W_{0,1}(x(z_0))\right)^{2k-2} \left[\vert_{w_1 = z_0}\right] \left[\vert_{w_2 = z_0}\right] \Delta_{w_1} \Delta_{w_2} \cW_{g,2,n}(w_1,w_2 \mid z_{\llbracket n \rrbracket})
\end{align}
is holomorphic at $ z_0 \to p$.
\end{lemma}
\begin{proof} This is a direct corollary of~\cite[corollary 3.4 and remark 3.3]{r-spinFullProof}. Note that the sum over $\alpha$s and $ \beta$s in the first line is the same as the sum over $\alpha$s in \eqref{eq:expr-r}, but split depending on whether $ \cS$ or $\Delta$ acts on the corresponding variable.
\end{proof}

Another statement that we need is the following. Let $f_i(z)$, $i\in \mathbb{Z}_{\geq 0}$ be a sequence of meromorphic functions defined on an open neighborhood $U$ of the point $p$ with the orders of poles  bounded by some constant. Assume $\sum_{i=0}^\infty f_i(z)$ converges absolutely and uniformly on every compact subset of $U\setminus \{p\}$ to a function of the same type, that is, to a meromorphic function $f(z)$ on $U$ with a possible pole only at the point $p$ with the order of the pole bounded by the same constant. Assume that we can split $\mathbb{Z}_{\geq 0}$ into a sequence of pairwise disjoint finite subsets $I_k$, $k=1,2,3,\dots$, $\mathbb{Z}_{\geq 0}=\bigsqcup_{k=1}^\infty I_k$, such that $\sum_{i\in I_k} f_i(z)$ is holomorphic at $z\to p$ for every $k$. Then we have:
\begin{lemma} \label{lem:cxanalysis} The sum $f(z)=\sum_{i=0}^\infty f_i(z)$ is holomorphic at $z\to p$. 
\end{lemma}
Now we are ready to prove proposition~\ref{prop:QLE}.

\begin{proof}[Proof of proposition~\ref{prop:QLE}] As stated before, the proof works by induction on $2g-2+(n+1)$. First, equation~\eqref{Sondiagonal} and the holomorphicity at $z_0\to p$ of the expression~\eqref{eq:holoinput} imply that the following expression is holomorphic  at $z_0\to p$: 
\begin{align}\label{eq:holoinput-2}
&\sum_{\substack{m \geq 1, d \geq 0 \\ m + 2d \geq 2}} 
\frac{(m+2d-1)!}{m!2^m} \cQ_{d,m}(z_0) \sum_{\substack{I\sqcup J = \llbracket m \rrbracket \\ |I|\in 2\mathbb{Z} }} \prod_{i\in I} \Delta_{w_i} \prod_{j\in J} \cS_{w_j} \cW_{g-d,m,n} (w_{\llbracket m \rrbracket} \mid z_{\llbracket n \rrbracket})
\\ \notag & 
+
\sum_{\alpha=1}^g c_\alpha
\sum_{\substack{m \geq 1, d \geq 0 \\ m + 2d \geq 1}} 
\frac{(m+2d-1+2\alpha)!}{m!2^m} \cQ_{d,m}(z_0) 
\sum_{\substack{I\sqcup J = \llbracket m \rrbracket \\ |I|\in 2\mathbb{Z} }} \prod_{i\in I} \Delta_{w_i} \prod_{j\in J} \cS_{w_j}
\cW_{g-d-\alpha,m,n} (w_{\llbracket m \rrbracket} \mid z_{\llbracket n \rrbracket})  .
\end{align}	
We split this expression into three parts and analyse them separately. 

The first part is the second summand. We consider
\begin{align}\label{eq:hologenusdefect}
&\sum_{\alpha=1}^g c_\alpha
\sum_{\substack{m \geq 1, d \geq 0 \\ m + 2d \geq 1}} 
\frac{(m+2d-1+2\alpha)!}{m!2^m} \cQ_{d,m}(z_0) 
\sum_{\substack{I\sqcup J = \llbracket m \rrbracket \\ |I|\in 2\mathbb{Z} }} \prod_{i\in I} \Delta_{w_i} \prod_{j\in J} \cS_{w_j}
\cW_{g-d-\alpha,m,n} (w_{\llbracket m \rrbracket} \mid z_{\llbracket n \rrbracket})  .
\end{align}	
Note that this expression is an infinite sum of the products of derivatives of the function $\{\tilde H_{g,n}\}_{g,n}$, and it absolutely uniformly converges  to a meromorphic function on $U\setminus \{p\}$ in the variable $z_0$, where $U$ is an open neighborhood of the point $p$. The proof of that is exactly the same as the proof of lemma~\ref{lem:convergence}: we have a finite number of terms with no factors of $\cS_{w} D_{\xi} \tilde H_{0,1}(\xi(w))$ and $\Delta_{w} D_{\xi} \tilde H_{0,1}(\xi(w))$ multiplied by a geometrically converging series in $\cS_{w} D_{\xi} \tilde H_{0,1}(\xi(w))$ and $\Delta_{w} D_{\xi} \tilde H_{0,1}(\xi(w))$. On the other hand, we can rewrite this expression as the sum over $r+1=m+2d$ and then for each fixed $r+1$ we have a finite expression, which is holomorphic at $z_0\to p$ according to lemma~\ref{lem:corollaryDKPS}, and using the induction hypothesis along with the fact that $ \alpha \geq 1$. Thus expression~\eqref{eq:hologenusdefect} satisfies the conditions of lemma~\ref{lem:cxanalysis}, and therefore \eqref{eq:hologenusdefect} converges to a holomorphic function on $U$. 

Introduce a new notation:
\begin{align}
\sum_{d \geq 0} \cQ^{\mathsf{red}}_{d,m}(z_0)t^{2d} & \coloneqq
\prod_{j = 1}^m \bigg( \left[\vert_{w_j = z_0}\right] \circ \frac{\zeta(tD_{x(w_j)})}{tD_{x(w_j)}}\bigg)\,.
\end{align}
This is the `leading order' part of $\cQ_{d,m}$ in the sense that it does not include the global derivatives of $ z_0$ or the extra $\frac{z}{\zeta (z)}$, which would lead to terms that have been shown to be holomorphic in earlier steps of the induction.
The second part is then all these extra terms, the ``genus defect'' part of the first summand in~\eqref{eq:holoinput-2}. We consider
\begin{align}\label{eq:holoinput-2-fisrtgendef}
&\sum_{\substack{m \geq 1, d \geq 0 \\ m + 2d \geq 2}} 
\frac{(m+2d-1)!}{m!2^m} \left(\cQ_{d,m}(z_0)- \cQ^{\mathsf{red}}_{d,m}(z_0)\right)\sum_{\substack{I\sqcup J = \llbracket m \rrbracket \\ |I|\in 2\mathbb{Z} }} \prod_{i\in I} \Delta_{w_i} \prod_{j\in J} \cS_{w_j} \cW_{g-d,m,n} (w_{\llbracket m \rrbracket} \mid z_{\llbracket n \rrbracket}).
\end{align}
Literally the same argument as in the case of expression~\eqref{eq:hologenusdefect} proves that~\eqref{eq:holoinput-2-fisrtgendef} converges to a  holomorphic function on $U$.

The third part is equal to 
\begin{align}\label{eq:finalpart}
&\sum_{\substack{m \geq 1, d \geq 0 \\ m + 2d \geq 2}} 
\frac{(m+2d-1)!}{m!2^m} \cQ^{\mathsf{red}}_{d,m}(z_0)\sum_{\substack{I\sqcup J = \llbracket m \rrbracket \\ |I|\in 2\mathbb{Z} }} \prod_{i\in I} \Delta_{w_i} \prod_{j\in J} \cS_{w_j} \cW_{g-d,m,n} (w_{\llbracket m \rrbracket} \mid z_{\llbracket n \rrbracket})
\\ \notag 
& =  \sum_{r=1}^\infty \frac{r!}{2^{r+1}}  \!\!\!\!\!\!\!\!\sum_{\substack{k,\alpha_1,\dots,\alpha_{2k}\\ \ell, \beta_1,\dots,\beta_\ell \\ 2k+\alpha_1+\cdots+\alpha_{2k} \\ +\ell + \beta_1+\cdots+\beta_{\ell} = r+1}} \!\!\!\!\!\!\!\! \frac{1}{\ell! (2k)!}
\\ \notag
& \phantom{=\ }
\prod_{i=1}^\ell \left[\vert_{w'_i = z_0}\right] \frac{D_{x(w'_i)}^{2\beta_i}}{(2\beta_i+1)!} \cS_{w'_i} 
\prod_{i=1}^{2k} \left[\vert_{w_i = z_0}\right] \frac{D_{x(w_i)}^{2\alpha_i}}{(2\alpha_i+1)!} \Delta_{w_i} 
 \cW_{g+2k+\ell-r-1,\ell+2k,n} (w'_{\llbracket \ell\rrbracket},w_{\llbracket 2k\rrbracket} \mid z_{\llbracket n\rrbracket}) \,,
\end{align}
and it must be holomorphic as $ z_0 \to p$, as it is the difference of equations~\eqref{eq:holoinput-2} and~\eqref{eq:hologenusdefect}, \eqref{eq:holoinput-2-fisrtgendef}.
Each of the $r$-summands of equation~\eqref{eq:finalpart} corresponds to the first part of the equation in lemma~\ref{lem:secondRSpinLemma}. By the same arguments as before, the sum over all $r$ of the expression in lemma~\ref{lem:secondRSpinLemma},
\begin{align}\label{eq:secondRSpinLemma2}
&\sum_{r= 1}^\infty
\frac{r!}{2^{r+1}} \mathsf{Expr}_r,
\end{align}
where $\mathsf{Expr}_r $ is equal to~\eqref{eq:expr-r},
still converges absolutely and uniformly on $U\setminus \{p\}$ in the variable $z_0$, and is holomorphic as $ z_0 \to p$.
Because of this, the difference between equations~\eqref{eq:finalpart} and~\eqref{eq:secondRSpinLemma2} must also be holomorphic. Explicitly, this is
\begin{align}\label{eq:QLEtimesInv}
& \sum_{\substack{k\geq 1\\ \ell \geq 0}}\frac{(2k+\ell -1)!}{2^{2k+\ell}\ell! (2k)!} \binom{k}{1} \left(\cS_{z_0} W_{0,1}(x(z_0))\right)^\ell \left(\Delta_{z_0} W_{0,1}(x(z_0))\right)^{2k-2} 
\\ \notag & 
\times\left[\vert_{w_1 = z_0}\right] \left[\vert_{w_2 = z_0}\right] \Delta_{w_1} \Delta_{w_2} \cW_{g,2,n}(w_1,w_2 \mid z_{\llbracket n \rrbracket}) \,.
\end{align}
To analyse this expression, let us first consider the sum
\begin{align}
\sum_{\substack{k\geq 1\\ \ell \geq 0}}\frac{(2k+\ell -1)!}{2^{2k+\ell}\ell! (2k)!} k \, s^\ell \delta^{2k-2} = \frac{1}{2\big((2-s)^2-\delta^2\big)}.
\end{align}
For $s=\cS_{z_0} W_{0,1}(x(z_0))$ and $\delta=\Delta_{z_0} W_{0,1}(x(z_0))$ both $(s+\delta)/2$ and $(s-\delta)/2$ belong to the unit ball for $z_0$ near $p$, as $ W_{0,1}(x(p)) = \frac{1}{q+1}$. Therefore, this expression defines a holomorphic function on $U$ in the variable $z_0$, non-vanishing at $z_0\to p$. This implies that 
\begin{equation}
\left[\vert_{w_1 = z_0}\right] \left[\vert_{w_2 = z_0}\right] \Delta_{w_1} \Delta_{w_2} \cW_{g,2,n}(w_1,w_2 \mid z_{\llbracket n \rrbracket})
\end{equation} is holomorphic at $z_0\to p$. 

Then equation~\eqref{Sondiagonal} and lemma~\ref{lem:lle} (in the case $g=0$, $n=0$ one also has to recall equation~\eqref{eq:redefinition02}) imply that $\cW_{g,2,n}(z_0,\bar{z}_0 \mid z_{\llbracket n \rrbracket})$ is holomorphic at $z_0\to p$ (cf.~also the arguments in~\cite[section 2.4]{BorotShadrin} and~\cite[section 3.2]{r-spinFullProof}).
\end{proof}

\bibliographystyle{alpha}
\bibliography{monotoneR}

\newcommand{\etalchar}[1]{$^{#1}$}
\begin{thebibliography}{DKPS19b}

\bibitem[ACEH18]{ACEH}
Alexander Alexandrov, Guillaume Chapuy, Bertrand Eynard, and John Harnad.
\newblock Weighted {H}urwitz numbers and topological recursion: an overview.
\newblock {\em J. Math. Phys.}, 59(8):081102, 21, 2018.

\bibitem[Ale11]{Alexandrov-CJ}
Alexander Alexandrov.
\newblock Matrix models for random partitions.
\newblock {\em Nuclear Phys. B}, 851(3):620--650, 2011.

\bibitem[ALS16]{ALS}
Alexander Alexandrov, Danilo Lewanski, and Sergey Shadrin.
\newblock {Ramifications of Hurwitz theory, KP integrability and quantum
  curves}.
\newblock {\em J. High Energy Phys.}, 2016(5):30, 2016.

\bibitem[BEO15]{BorotEynardOrantin}
Ga{\"{e}}tan Borot, Bertrand Eynard, and Nicolas Orantin.
\newblock Abstract loop equations, topological recursion and new applications.
\newblock {\em Commun. Number Theory Phys.}, 9(1):51--187, 2015.

\bibitem[BKL{\etalchar{+}}17]{BKLPS}
Ga{\"e}tan {Borot}, Reinier {Kramer}, Danilo {Lewanski}, Alexandr {Popolitov},
  and Sergey {Shadrin}.
\newblock {Special cases of the orbifold version of Zvonkine's $r$-ELSV
  formula}.
\newblock {\em arXiv e-prints}, page arXiv:1705.10811, May 2017.

\bibitem[BKMP09]{BKMP}
Vincent Bouchard, Albrecht Klemm, Marcos {Mari\~{n}o}, and Sara Pasquetti.
\newblock Remodeling the {B}-model.
\newblock {\em Comm. Math. Phys.}, 287(1):117--178, 2009.

\bibitem[BS17]{BorotShadrin}
Ga{\"{e}}tan Borot and Sergey Shadrin.
\newblock Blobbed topological recursion: properties and applications.
\newblock {\em Math. Proc. Camb. Phil. Soc.}, 162:39--87, 2017.

\bibitem[CE06]{ChekhovEynard}
Leonid Chekhov and Bertrand Eynard.
\newblock Hermitian matrix model free energy: {F}eynman graph technique for all
  genera.
\newblock {\em J. High Energy Phys.}, (3):014, 18, 2006.

\bibitem[DDM17]{DoDyerMathews}
Norman Do, Alastair Dyer, and Daniel~V. Mathews.
\newblock {Topological recursion and a quantum curve for monotone Hurwitz
  numbers}.
\newblock {\em J. Geom. Phys.}, 120:19--36, 2017.

\bibitem[DK17]{DoKarev}
Norman Do and Maksim Karev.
\newblock Monotone orbifold {H}urwitz numbers.
\newblock {\em J. Math. Sci. (NY)}, 226(5):568--587, 2017.

\bibitem[DKPS19a]{DBKraPS}
Petr {Dunin-Barkowski}, Reinier Kramer, Alexandr Popolitov, and Sergey Shadrin.
\newblock Cut-and-join equation for monotone {H}urwitz numbers revisited.
\newblock {\em J. Geom. Phys.}, 137:1--6, 2019.

\bibitem[DKPS19b]{r-spinFullProof}
Petr {Dunin-Barkowski}, Reinier {Kramer}, Alexandr {Popolitov}, and Sergey
  {Shadrin}.
\newblock {Loop equations and a proof of Zvonkine's $qr$-ELSV formula}.
\newblock {\em arXiv e-prints}, page arXiv:1905.04524, May 2019.

\bibitem[DLPS15]{DLPS}
Petr {Dunin-Barkowski}, Danilo {Lewanski}, Alexander {Popolitov}, and Sergey
  {Shadrin}.
\newblock {Polynomiality of orbifold Hurwitz numbers, spectral curve, and a new
  proof of the Johnson-Pandharipande-Tseng formula.}
\newblock {\em {J. Lond. Math. Soc., II. Ser.}}, 92(3):547--565, 2015.

\bibitem[DOSS14]{DOSS}
Petr {Dunin-Barkowski}, Nicolas Orantin, Sergey Shadrin, and Loek Spitz.
\newblock Identification of the {G}ivental formula with the spectral curve
  topological recursion procedure.
\newblock {\em Comm. Math. Phys.}, 328(2):669--700, 2014.

\bibitem[EO07]{EynardOrantin}
Bertrand Eynard and Nicolas Orantin.
\newblock {Invariants of algebraic curves and topological expansion}.
\newblock {\em Communications in Number Theory and Physics}, 1(2):347--452,
  2007.

\bibitem[Eyn14]{Eynard}
Bertrand Eynard.
\newblock Invariants of spectral curves and intersection theory of moduli
  spaces of complex curves.
\newblock {\em Commun. Number Theory Phys.}, 8(3):541--588, 2014.

\bibitem[Eyn16]{EynardBook}
Bertrand Eynard.
\newblock {\em Counting surfaces}, volume~70 of {\em Progress in Mathematical
  Physics}.
\newblock Birkh\"{a}user/Springer, 2016.
\newblock CRM Aisenstadt chair lectures.

\bibitem[GGN13a]{GouldenetalGenus0}
Ian~P. Goulden, Mathieu {Guay-Paquet}, and Jonathan Novak.
\newblock {Monotone Hurwitz numbers in genus zero}.
\newblock {\em Canad. J. Math.}, 65(5):1020--1042, 2013.

\bibitem[GGN13b]{GouldenPolynomiality}
Ian~P. Goulden, Mathieu {Guay-Paquet}, and Jonathan Novak.
\newblock {Polynomiality of monotone Hurwitz numbers in higher genera}.
\newblock {\em Adv. Math.}, 238:1--23, 2013.

\bibitem[GGN14]{GouldenHCIZ}
Ian~P. Goulden, Mathieu {Guay-Paquet}, and Jonathan Novak.
\newblock {Monotone Hurwitz numbers and the HCIZ integral}.
\newblock {\em Ann. Math. Blaise Pascal}, 21(1):71--89, 2014.

\bibitem[GH15]{GuayHarnad}
Mathieu {Guay-Paquet} and John Harnad.
\newblock 2{D} {T}oda {$\tau$}-functions as combinatorial generating functions.
\newblock {\em Lett. Math. Phys.}, 105(6):827--852, 2015.

\bibitem[Hah19]{Hahn}
Marvin~Anas Hahn.
\newblock A monodromy graph approach to the piecewise polynomiality of simple,
  monotone and {G}rothendieck dessins d'enfants double {H}urwitz numbers.
\newblock {\em Graphs Combin.}, 35(3):729--766, 2019.

\bibitem[HKL18]{HahnKramerLewanski}
Marvin~Anas Hahn, Reinier Kramer, and Danilo Lewanski.
\newblock Wall-crossing formulae and strong piecewise polynomiality for mixed
  {G}rothendieck dessins d'enfant, monotone, and double simple {H}urwitz
  numbers.
\newblock {\em Adv. Math.}, 336:38--69, 2018.

\bibitem[HO15]{HarnadOrlov}
John Harnad and Aleksander~Yu. Orlov.
\newblock {Hypergeometric \texorpdfstring{\( \tau\)}{tau}-functions, Hurwitz
  numbers and enumeration of paths}.
\newblock {\em Comm. Math. Phys.}, 338(1):267--284, 2015.

\bibitem[Kac90]{KacBook}
Victor~G. Kac.
\newblock {\em Infinite-dimensional {L}ie algebras}.
\newblock Cambridge University Press, Cambridge, third edition, 1990.

\bibitem[KLS19]{KLS}
Reinier {Kramer}, Danilo {Lewa\'nski}, and Sergey {Shadrin}.
\newblock {Quasi-polynomiality of monotone orbifold Hurwitz numbers and
  Grothendieck's dessins d'enfants.}
\newblock {\em {Doc. Math.}}, 24:857--898, 2019.

\bibitem[LM18]{LiuMulase}
Chiu-Chu~Melissa Liu and Motohico Mulase, editors.
\newblock {\em Topological recursion and its influence in analysis, geometry,
  and topology}, volume 100 of {\em Proceedings of Symposia in Pure
  Mathematics}.
\newblock American Mathematical Society, Providence, RI, 2018.
\newblock 2016 AMS von Neumann Symposium Topological Recursion and its
  Influence in Analysis, Geometry, and Topology, July 4--8, 2016, Charlotte,
  North Carolina.

\bibitem[MJD00]{MiwaJimboDate}
T.~Miwa, M.~Jimbo, and E.~Date.
\newblock {\em Solitons}, volume 135 of {\em Cambridge Tracts in Mathematics}.
\newblock Cambridge University Press, Cambridge, 2000.
\newblock Differential equations, symmetries and infinite-dimensional algebras,
  Translated from the 1993 Japanese original by Miles Reid.

\bibitem[Ros08]{Rossi-CJ}
Paolo Rossi.
\newblock Gromov-{W}itten invariants of target curves via symplectic field
  theory.
\newblock {\em J. Geom. Phys.}, 58(8):931--941, 2008.

\bibitem[SSZ12]{SSZ-LMS}
Sergey Shadrin, Loek Spitz, and Dimitri Zvonkine.
\newblock On double {H}urwitz numbers with completed cycles.
\newblock {\em J. Lond. Math. Soc. (2)}, 86(2):407--432, 2012.

\end{thebibliography}

\end{document}